\documentclass{amsart}%
\usepackage{amsfonts}
\usepackage{amsmath}
\usepackage{amssymb}
\usepackage{graphicx}%
\newtheorem{theorem}{Theorem}
\theoremstyle{plain}

\newtheorem{corollary}{Corollary}

\newtheorem{lemma}{Lemma}

\newtheorem{proposition}{Proposition}
\newtheorem{remark}{Remark}

\numberwithin{equation}{section}

\begin{document}
\title[volume entropy and rigidity]{An integral formula for the volume entropy with applications to rigidity}
\author{Fran\c{c}ois Ledrappier}
\address{LPMA, C.N.R.S. UMR7599, Bo\^ite Courrier 188 - 4, Place Jussieu\\
75252 Paris cedex 05, France}
\email{fledrapp@nd.edu}
\author{Xiaodong Wang}
\address{Department of Mathematics,\\
Michigan State University\\
East Lansing, MI 48824, USA}
\email{xwang@math.msu.edu}
\thanks{The first author was partially supported by NSF grant DMS-0801127}
\thanks{The second author was partially supported by NSF grant DMS-0905904}
\date{Nov 2, 2009}
\maketitle

\section{\bigskip Introduction}

Let $M^{n}$ be a compact Riemannian manifold and $\pi:\widetilde{M}\rightarrow
M$ its universal covering. The fundamental group $G=\pi_{1}\left(  M\right)  $
acts on $\widetilde{M}$ as isometries such that $M=$ $\widetilde{M}/G$.
Associated to $\widetilde{M}$ are several asymptotic invariants. In this paper
we are primarily concerned with the volume entropy $v$ defined by%
\[
v=\lim_{r\rightarrow\infty}\frac{\ln\mathrm{vol}B_{\widetilde{M}}\left(
x,r\right)  }{r},
\]
where $B_{\widetilde{M}}\left(  x,r\right)  $ is the ball of radius $r$
centered at $x$ in $\widetilde{M}$. It is proved by Manning \cite{M} and
Freire-Ma\~n\'e \cite{FM} that

\begin{itemize}
\item the limit exists and is independent of the center $x\in\widetilde{M}$,

\item $v\leq H$, the topological entropy of the geodesic flow on $M$,

\item $v=H$ if $M$ has no conjugate points.
\end{itemize}

There has been a lot of work on understanding the volume entropy of which we
only mention the celebrated paper of Besson, Courtois and Gallot \cite{BCG1}
where one can find other results and references. But the volume entropy still
remains a subtle invariant. If $M$ is negatively curved, it is better
understood due to the existence of the so called Patterson-Sullivan measure on
the ideal boundary. Let $\partial\widetilde{M}$ be the ideal boundary of
$\widetilde{M}\,$\ defined as equivalence classes of geodesic rays. We fix a
base point $o\in\widetilde{M}$ and for $\xi\in\partial\widetilde{M}$ we denote
$B_{\xi}$ the associated Busemann function, i.e.
\[
B_{\xi}\left(  x\right)  =\lim_{t\rightarrow\infty}d\left(  x,\gamma\left(
t\right)  \right)  -t,
\]
where $\gamma$ is the geodesic ray initiating from $o$ representing $\xi$. It
is well known that $B_{\xi}$ is smooth and its gradient is of length one. The
Patterson-Sullivan measure \cite{P, S, K} is a family $\left\{  \nu_{x}%
:x\in\widetilde{M}\right\}  $ of measures on $\partial\widetilde{M}$ s.t.:

\begin{itemize}
\item for any pair $x,y\in\widetilde{M}$, the two measures $\nu_{x},\nu_{y}$
are equivalent with%
\[
\frac{d\nu_{x}}{d\nu_{y}}\left(  \xi\right)  =e^{-v\left(  B_{\xi}\left(
x\right)  -B_{\xi}\left(  y\right)  \right)  };
\]

\item for any $g\in G$%
\[
g_{\ast}\nu_{x}=\nu_{gx}.
\]

\end{itemize}

The Patterson-Sullivan measure contains a lot of information and plays an
important role in \cite{BCG2}. Moreover, it is proved by Knieper, Ledrappier
and Yue (\cite{K, L2, Y1}) that the following integral formula for the volume
entropy holds in terms of the Patterson-Sullivan measure:
\begin{equation}
v=\frac{1}{C}\int_{M}\left(  \int_{\partial\widetilde{M}}\Delta B_{\xi}\left(
x\right)  d\nu_{x}\left(  \xi\right)  \right)  dx, \label{fvn}%
\end{equation}
where $C=\int_{M}\nu_{x}\left(  \partial\widetilde{M}\right)  dx$. (To
interpret the formula properly, notice after integrating over $\partial
\widetilde{M}$ we get a function on $\widetilde{M}$ which is $G$-invariant and
hence descends to $M$.) This formula shows how $v$ interacts with local geometry.

In this paper, we will extend the theory of Patterson-Sullivan measure to any
manifold without the negative curvature assumption. More generally, let
$\pi:\widetilde{M}\rightarrow M$ be a regular Riemannian covering of a compact
manifold $M$ and $G$ the discrete group of deck transformations. \ We will
consider the Busemann compactification of $\widetilde{M}$, denoted by
$\widehat{M}$. On the Busemann boundary $\partial\widehat{M}$ we will
construct Patterson-Sullivan measure which retains the essential features of
the classical theory. Namely

\begin{theorem}
\label{div} There exists a probability measure $\nu$ on the laminated space
$X_{M}=\left(  \widetilde{M}\times\partial\widehat{M}\right)  /G$ such that
for any continuous vector field $Y$ on $X_{M}$ which is $C^{1}$ along the
leaves,%
\[
\int\mathrm{div}^{\mathcal{W}}Yd\nu=v\int\left\langle Y,\nabla^{\mathcal{W}%
}\xi\right\rangle d\nu,
\]
where $\mathrm{div}^{\mathcal{W}}$ and $\nabla^{\mathcal{W}}$ are laminated
divergence and gradient, respectively.
\end{theorem}

As an application of the above theorem, we will prove the following rigidity theorem.

\begin{theorem}
\label{Rig}Let $M^{n}$ be a compact Riemannian manifold with $Ric\geq-\left(
n-1\right)  $ and $\pi:\widetilde{M}\rightarrow M$ a regular covering. Then
the volume entropy of $\widetilde{M}$ satisfies $v\leq\left(  n-1\right)  $
and equality holds iff $M$ is hyperbolic.
\end{theorem}

The inequality $v\leq\left(  n-1\right)  $ is of course well known and follows
easily from the volume comparison theorem. What is new is the rigidity part.
To have some perspective on this result, recall another invariant: the bottom
spectrum of the Laplacian on $\widetilde{M}$, denoted by $\lambda_{0}$ and
defined as%
\[
\lambda_{0}=\inf_{f\in C_{c}^{1}\left(  \widetilde{M}\right)  }\frac
{\int_{\widetilde{M}}\left\vert \nabla f\right\vert ^{2}}{\int_{\widetilde{M}%
}f^{2}}.
\]
It is a well known fact that $\lambda_{0}\leq v^{2}/4$. Therefore as an
immediate corollary of Theorem \ref{Rig} we have the following result
previously proved by the second author \cite{W}.

\begin{corollary}
Let $\left(  M^{n},g\right)  $ be a compact Riemannian manifold with
$Ric\geq-\left(  n-1\right)  $ and $\pi:\widetilde{M}\rightarrow M$ a regular
covering. If $\lambda_{0}=\left(  n-1\right)  ^{2}/4$, then $\widetilde{M}$ is
isometric to the hyperbolic space $\mathbb{H}^{n}$.
\end{corollary}

Clearly the asymptotic invariant $v$ is much weaker than $\lambda_{0}$. It is
somewhat surprising that we still have a rigidity theorem for $v$. If $M$ is
negatively curved, Theorem \ref{Rig} is proved by Knieper \cite{K} using
(\ref{fvn}). The proof in the general case is more subtle due to the fact the
Busemann functions are only Lipschitz. In fact, it is partly to prove this
rigidity result that we are led to the construction of the measure $\nu$ and
the formula in Theorem \ref{div}.

We will also discuss the K\"{a}hler and quaternionic K\"{a}hler analogue of
Theorem \ref{Rig}. In the K\"{a}hler case our method yields the following

\begin{theorem}
\label{Ka}Let $M$ be a compact K\"{a}hler manifold with $\dim_{\mathbb{C}}M=m$
and $\pi:\widetilde{M}\rightarrow M$ a regular covering. If the bisectional
curvature $K_{\mathbb{C}}\geq-2$, then the volume entropy $v$ satisfies
$v\leq2m$. Moreover equality holds iff $M$ is complex hyperbolic (normalized
to have constant holomorphic sectional curvature $-4$).
\end{theorem}

To clarify the statement, the condition $K_{\mathbb{C}}\geq-2$ means that for
any two vectors $X,Y$%
\[
R\left(  X,Y,X,Y\right)  +R\left(  X,JY,X,JY\right)  \geq-2\left(  \left\vert
X\right\vert ^{2}\left\vert Y\right\vert ^{2}+\left\langle X,Y\right\rangle
^{2}+\left\langle X,JY\right\rangle ^{2}\right)  ,
\]
where $J$ is the complex structure.

In the quaternionic K\"{a}hler case we have

\begin{theorem}
\label{QK}Let $M$ be a compact quaternionic K\"{a}hler manifold of $\dim=4m$
with $m\geq2$ and scalar curvature $-16m\left(  m+2\right)  $. Let
$\pi:\widetilde{M}\rightarrow M$ be a regular covering. Then the volume
entropy $v$ satisfies $v\leq2\left(  2m+1\right)  $. Moreover equality holds
iff $M$ is quaternionic hyperbolic.
\end{theorem}

The paper is organized as follows. In Section 2, we discuss the Busemann
compactification and construct the Patterson-Sullivan measure and prove
Theorem \ref{div}. Theorem \ref{Rig} will be proved in Section 3. We will
discuss the K\"{a}hler case and the quaternionic K\"{a}hler case in Section 4.

\section{Construction of the measure}

Let $\widetilde{M}$ be a noncompact, complete Riemannian manifold. Fix a point
$o\in\widetilde{M}$ and define, for $x\in\widetilde{M}$ the function $\xi
_{x}(z)$ on $\widetilde{M}$ by:
\[
\xi_{x}(z)\;=\;d(x,z)-d(x,o).
\]
The assignment $x\mapsto\xi_{x}$ is continuous, one-to-one and takes values in
a relatively compact set of functions for the topology of uniform convergence
on compact subsets of $\widetilde{M}$. The Busemann compactification
$\widehat{M}$ of $\widetilde{M}$ is the closure of $\widetilde{M}$ for that
topology. The space $\widehat{M}$ is a compact separable space. The
\textit{Busemann boundary } $\partial\widehat{M}:=\widehat{M}\setminus
\widetilde{M}$ is made of Lipschitz continuous functions $\xi$ on
$\widetilde{M}$ such that $\xi(o)=0$. Elements of $\partial\widehat{M}$ are
called \textit{horofunctions}.

First we collect some general facts about horofunctions, see e.g. \cite{SY,
Pe}. Suppose $\xi\in\widehat{M}$ is the limit of $\left\{  a_{k}\right\}
\subset\widetilde{M}$ with $d\left(  o,a_{k}\right)  \rightarrow\infty$, i.e.
\begin{equation}
\xi\left(  x\right)  =\lim_{k\rightarrow\infty}f_{k}\left(  x\right)  ,
\label{lim}%
\end{equation}
where $f_{k}\left(  x\right)  =\xi_{a_{k}}\left(  x\right)  =d\left(
x,a_{k}\right)  -d\left(  o,a_{k}\right)  $. The convergence is uniform over
compact sets. We fix a point $p\in\widetilde{M}$ and for each $k$ let
$\gamma_{k}$ be a minimizing geodesic from $p$ to $a_{k}$. Passing to a
subsequence, we can assume that $\gamma_{k}$ converges to a geodesic ray
$\gamma$ starting from $p$. Let $b_{\gamma}$ be the Busemann function
associated to $\gamma$, i.e. $b_{\gamma}\left(  x\right)  =\lim_{s\rightarrow
+\infty}d\left(  x,\gamma\left(  s\right)  \right)  -s$.

\begin{lemma}
\label{BF}\bigskip We have

\begin{enumerate}
\item $\xi\circ\gamma\left(  s\right)  =\xi\left(  p\right)  -s$ for $s\geq0$;

\item $\xi\left(  x\right)  \leq\xi\left(  p\right)  +d\left(  x,\gamma\left(
s\right)  \right)  -s$ for $s\geq0$;

\item $\xi\left(  x\right)  \leq\xi\left(  p\right)  +b_{\gamma}\left(
x\right)  $.
\end{enumerate}
\end{lemma}

\begin{proof}
For any $s>0$ and $\varepsilon>0$ we have $d\left(  \gamma_{k}\left(
s\right)  ,\gamma\left(  s\right)  \right)  \leq\varepsilon$ for $k$
sufficiently large. Then
\begin{align*}
f_{k}\circ\gamma\left(  s\right)  -f_{k}\left(  p\right)   &  =d\left(
\gamma\left(  s\right)  ,a_{k}\right)  -d\left(  p,a_{k}\right) \\
&  =d\left(  \gamma\left(  s\right)  ,a_{k}\right)  -d\left(  \gamma
_{k}\left(  s\right)  ,a_{k}\right)  +d\left(  \gamma_{k}\left(  s\right)
,a_{k}\right)  -d\left(  p,a_{k}\right) \\
&  \leq d\left(  \gamma\left(  s\right)  ,\gamma_{k}\left(  s\right)  \right)
+d\left(  \gamma_{k}\left(  s\right)  ,a_{k}\right)  -d\left(  p,a_{k}\right)
\\
&  =d\left(  \gamma\left(  s\right)  ,\gamma_{k}\left(  s\right)  \right)
-s\\
&  \leq\varepsilon-s.
\end{align*}
Taking limit yields $\xi\circ\gamma\left(  s\right)  -\xi\left(  p\right)
\leq\varepsilon-s$. Hence $\xi\circ\gamma\left(  s\right)  \leq\xi\left(
p\right)  -s$. On the other hand we have the reversed inequality $\xi
\circ\gamma\left(  s\right)  \geq\xi\left(  p\right)  -s$ as $\xi$ is
Lipschitz with Lipschitz constant $1$.

To prove the second part, we have for $s\geq0$%
\begin{align*}
f_{k}\left(  x\right)   &  =d\left(  x,a_{k}\right)  -d\left(  o,a_{k}\right)
\\
&  \leq d\left(  x,\gamma\left(  s\right)  \right)  +d\left(  a_{k}%
,\gamma\left(  s\right)  \right)  -d\left(  o,a_{k}\right)  .
\end{align*}
Letting $k\rightarrow\infty$ yields%
\begin{align*}
\xi\left(  x\right)   &  \leq d\left(  x,\gamma\left(  s\right)  \right)
+\xi\circ\gamma\left(  s\right) \\
&  =d\left(  x,\gamma\left(  s\right)  \right)  -s+\xi\left(  p\right)  .
\end{align*}
Taking limit as $s\rightarrow\infty$ yields the third part.
\end{proof}

It follows that if $\xi$ is differentiable at $x$, then $\left\vert \nabla
\xi\left(  x\right)  \right\vert =1$. Therefore $\left\vert \nabla
\xi\right\vert =1$ almost everywhere on $\widetilde{M}$.

\begin{proposition}
$\widetilde{M}$ is open in its Busemann compactification $\widehat{M}$. Hence
the Busemann boundary $\partial\widehat{M}$ is compact.
\end{proposition}

\begin{proof}
Suppose otherwise and $p\in\widetilde{M}$ is the limit of a sequence $\left\{
a_{k}\right\}  \subset\widetilde{M}$ with $d\left(  o,a_{k}\right)
\rightarrow\infty$, i.e. $\xi_{p}\left(  x\right)  =\lim_{k\rightarrow\infty
}\xi_{a_{k}}\left(  x\right)  $ and the convergence is uniform over compact
sets. Then by Lemma \ref{BF} there is a geodesic ray $\gamma$ starting from
$p$ s.t. $\xi_{p}\circ\gamma\left(  s\right)  =\xi_{p}\left(  p\right)
-s=-s-d\left(  o,p\right)  $ for $s\geq0$. But
\begin{align*}
\xi_{p}\circ\gamma\left(  s\right)   &  =d\left(  \gamma\left(  s\right)
,p\right)  -d\left(  o,p\right) \\
&  =s-d\left(  o,p\right)  ,
\end{align*}
Clearly a contradiction.
\end{proof}

We now further assume that $\widetilde{M}$ is a regular Riemannian covering of
a compact manifold $M$, i.e. $\widetilde{M}$ is a Riemannian manifold and
there is a discrete group $G$ of isometries of $\widetilde{M}$ acting freely
and such that the quotient $M=\widetilde{M}/G$ is a compact manifold. The
quotient metric makes $M$ a compact Riemannian manifold. We recall the
construction of the laminated space $X_{M}$ (\cite{L1}). Observe that we may
extend by continuity the action of $G$ from $\widetilde{M}$ to $\widehat{M}$,
in such a way that for $\xi$ in $\widehat{M}$ and $g$ in $G$,
\[
g.\xi(z)\;=\;\xi(g^{-1}z)-\xi(g^{-1}o).
\]
We define now the \textit{horospheric suspension } $X_{M}$ of $M$ as the
quotient of the space $\widetilde{M}\times\widehat{M}$ by the diagonal action
of $G$. The projection onto the first component in $\widetilde{M}%
\times\widehat{M}$ factors into a projection from $X_{M}$ to $M$ so that the
fibers are isometric to $\widehat{M}$. It is clear that the space $X_{M}$ is
metric compact. If $M_{0}\subset\widetilde{M}$ is a fundamental domain for
$M$, one can represent $X_{M}$ as $M_{0}\times\widehat{M}$ in a natural way.

To each point $\xi\in\widehat{M}$ is associated the projection $W_{\xi}$ of
$\widetilde{M} \times\{\xi\}$. As a subgroup of $G$, the stabilizer $G_{\xi}$
of the point $\xi$ acts discretely on $\widetilde{M}$ and the space $W_{\xi}$
is homeomorphic to the quotient of $\widetilde{M}$ by $G_{\xi}$. We put on
each $W_{\xi}$ the smooth structure and the metric inherited from
$\widetilde{M}$. The manifold $W_{\xi}$ and its metric vary continuously on
$X_{M}$. The collection of all $W_{\xi}, \xi\in\widehat{M}$ form a continuous
lamination $\mathcal{W}_{M}$ with leaves which are manifolds locally modeled
on $\widetilde{M}$. In particular, it makes sense to differentiate along the
leaves of the lamination and we denote $\nabla^{\mathcal{W}} $ and
$\mathrm{div}^{\mathcal{W}}$ the associated gradient and divergence operators:
$\nabla^{\mathcal{W}}$ acts on continuous functions which are $C^{1}$ along
the leaves of $\mathcal{W}$, $\mathrm{div}^{\mathcal{W}}$ on continuous vector
fields in $T\mathcal{W}$ which are of class $C^{1} $ along the leaves of
$\mathcal{W}$. We want to construct a measure on $X_{M}$ which would behave as
the Knieper measure $\nu= \int_{M} \left(  \int_{\partial\widetilde{M}}
d\nu_{x} \left(  \xi\right)  \right)  dx $ in the negatively curved case. The
construction follows Patterson's in the Fuchsian case.

Let $v$ be the volume entropy of $\widetilde{M}$%
\[
v=\lim_{r\rightarrow\infty}\frac{\ln\mathrm{vol}B_{\widetilde{M}}\left(
x,r\right)  }{r},
\]
where $B_{\widetilde{M}}\left(  x,r\right)  $ is the ball of radius $r$
centered at $x$ in $\widetilde{M}$.

Let us consider the Poincar\'{e} series of $\widetilde{M}$:%
\[
P(s):=\sum_{g\in G}e^{-sd\left(  o,go\right)  }.
\]

\begin{proposition}
The series $P(s)$ converges for $s>v$, diverges to $+\infty$ for $s<v$.
\end{proposition}

\begin{proof}
This is classical and for completeness we recall the proof. Take $M_{0}$ a
fundamental domain in $\widetilde{M}$ containing $o$ in its interior, and
positive constants $d,D$ such that $B(o,d)\subset M_{0}\subset B(o,D)$.

Define $\pi(R):=\sharp\left\{  g\in G:d\left(  o,go\right)  \leq R\right\}  $.
We have $\pi(R+S)\leq\pi(R+D)\pi(S+D)$, which implies $\pi(R+S+2D)\leq
\pi(R+2D)\pi(S+2D)$. It follows that the following limit exists
\[
\lim_{R\rightarrow\infty}\frac{1}{R}\ln\pi(R+2D)=\inf_{R}\frac{1}{R}\ln
\pi(R+2D).
\]
The above limit is the critical exponent of the Poincar\'e series. Since
\[
\pi(R)\mathrm{vol}B(o,d)\leq\mathrm{vol}B(o,R+d)\leq\pi(R+D)\mathrm{vol}M ,
\]
the above limit is also $\lim_{R\rightarrow\infty}\frac{\ln\mathrm{vol}%
B_{\widetilde{M}}\left(  x,R\right)  }{R} = v$.
\end{proof}

\bigskip As in the classical case, a distinction has to be made between the
case that $P\left(  s\right)  $ diverges at $v$ and the case that it
converges. The following lemma is due to Patterson \cite{P} (see also \cite{N}).

\begin{lemma}
There exists a function $h:\mathbb{R}^{+}\rightarrow\mathbb{R}^{+}$ which is
continuous, non-decreasing, \ and

\begin{enumerate}
\item the series $P^{\ast}\left(  s\right)  : =\sum_{g\in G}e^{-sd\left(
o,go\right)  }h\left(  e^{d\left(  o,go\right)  }\right)  $ converges for
$s>v$ and diverges for $s\leq v$;

\item if $\varepsilon>0$ is given there exists $r_{0}$ s.t. for $r>r_{0},t>1
$, $h\left(  rt\right)  \leq t^{\varepsilon}h\left(  r\right)  $.
\end{enumerate}
\end{lemma}

\bigskip If $P\left(  s\right)  $ diverges at $v$ we will simply take $h$ to
be identically $1$. As a consequence of property (2) above we note that for
$t$ in a bounded interval%
\[
\frac{h\left(  e^{r+t}\right)  }{h\left(  e^{r}\right)  }\rightarrow1
\]
uniformly as $r\rightarrow\infty$.

For $x\in\widetilde{M},s>v$, we define a finite measure $\nu_{x,s}$ by
setting, for all $f$ continuous on $\widehat{M}$,%

\[
\int f(\xi)d\nu_{x,s}(\xi):=\frac{1}{P^{\ast}\left(  s\right)  }\sum_{g\in
G}e^{-sd(x,go)}h\left(  e^{d\left(  x,go\right)  }\right)  f(\xi_{go}).
\]

Clearly, for $g\in G$, $g_{\ast}\nu_{x,s}=\nu_{gx,s}$, so that the measure
$\widetilde{\nu}_{s}:=\int\nu_{x,s}dx$ is $G$-invariant on $\widetilde
{M}\times\widehat{M}$. We write $\nu_{s}$ for the corresponding measure on
$X_{M}=\widetilde{M}\times\widehat{M}/G$. Choose a sequence $s_{k}$ $>v$ and
$s_{k}\rightarrow v$ as $k\rightarrow\infty$ such that the probability
measures $\nu_{o,s_{k}}$ converge towards some probability measure $\nu_{o}$.
Since $\lim_{k\rightarrow\infty}P^{\ast}\left(  s_{k}\right)  =\infty$, the
measure $\nu_{o}$ is supported on $\partial\widehat{M}$.

\begin{proposition}
\bigskip For any $x\in\widetilde{M}$, the measures $\nu_{x,s_{k}}$ converge to
a measure $\nu_{x}$ on $\partial\widehat{M}$. Moreover%
\[
d\nu_{x}\left(  \xi\right)  =e^{-v\xi\left(  x\right)  }d\nu_{o}\left(
\xi\right)  .
\]

\end{proposition}

\bigskip In particular, for any $g\in G$ we have%
\[
d\left(  g_{\ast}\nu_{o}\right)  \left(  \xi\right)  =d\nu_{go}\left(
\xi\right)  =e^{-v\xi\left(  go\right)  }d\nu_{o}\left(  \xi\right)  , \quad
d\left(  g_{\ast}\nu_{x}\right)  \left(  \xi\right)  =d\nu_{gx}\left(
\xi\right)
\]
and the limit of the measures $\nu_{s_{k}}$ on $X_{M}$ is a measure $\nu$ on
$X_{M}$ which can be written, in the $M_{0}\times\widehat{M}$ representation
of $X_{M}$, as%

\begin{equation}
\nu=e^{-v\xi\left(  x\right)  }d\nu_{o}\left(  \xi\right)  dx \label{abs.cont}%
\end{equation}

\begin{proof}
Observe first that for a fixed $x$, $\nu_{x,s}(\widehat{M}) \leq e^{(v+s)
d(o,x)}$, so that the $\nu_{x,s}$ form a bounded family of measures on
$\widehat{M}$. Let $f$ be a continuous function on $\widehat{M}$. We may write:%

\begin{align*}
&  \int f\left(  \xi\right)  e^{-v\xi\left(  x\right)  }d\nu_{o}\left(
\xi\right) \\
&  =\lim_{k\rightarrow\infty}\int f\left(  \xi\right)  e^{-v\xi\left(
x\right)  }d\nu_{o,s_{k}}\left(  \xi\right) \\
&  =\lim_{k\rightarrow\infty}\frac{1}{P^{\ast}\left(  s_{k}\right)  }
\sum_{g\in G}f\left(  \xi_{go}\right)  e^{-v\left(  d\left(  x,go\right)
-d\left(  o,go\right)  \right)  }e^{-s_{k}d\left(  o,go\right)  }h\left(
e^{d\left(  o,go\right)  }\right) \\
&  =\lim_{k\rightarrow\infty}\frac{1}{P^{\ast}\left(  s_{k}\right)  }
\sum_{g\in G}f\left(  \xi_{go}\right)  e^{\left(  s_{k}-v\right)  \xi
_{go}\left(  x\right)  }\frac{h\left(  e^{d\left(  o,go\right)  }\right)
}{h\left(  e^{d\left(  x,go\right)  }\right)  }e^{-s_{k}d\left(  x,go\right)
}h\left(  e^{d\left(  x,go\right)  }\right) \\
&  =\lim_{k\rightarrow\infty}\left(  \int f\left(  \xi\right)  e^{\left(
s_{k}-v\right)  \xi\left(  x\right)  }d\nu_{x,s_{k}}\left(  \xi\right)  \; +
\; \varepsilon_{k}\right)  ,
\end{align*}
where
\[
\varepsilon_{k}=\frac{1}{P^{\ast}\left(  s_{k}\right)  }\sum_{g\in G}f\left(
x,\xi_{go}\right)  e^{\left(  s_{k}-v\right)  \xi_{go}\left(  x\right)
}\left(  \frac{h\left(  e^{d\left(  o,go\right)  }\right)  }{h\left(
e^{d\left(  x,go\right)  }\right)  }-1\right)  e^{-s_{k}d\left(  x,go\right)
}h\left(  e^{d\left(  x,go\right)  }\right)  .
\]
Suppose $\lim\varepsilon_{k} = 0$. Then, for a fixed $x$, $e^{\left(
s_{k}-v\right)  \xi\left(  x\right)  }$ converges to $1$ and therefore, the
limit exists and is $\int f e^{-v\xi(x)} d\nu_{o}$, as claimed. It only
remains to show that $\lim\varepsilon_{k}=0$. Indeed, for any $\delta>0$ and
any $x$, there exists a finite set $E\subset G$ s.t. for any $g\in G\backslash
E$
\[
\left\vert \frac{h\left(  e^{d\left(  o,go\right)  }\right)  }{h\left(
e^{d\left(  x,go\right)  }\right)  }-1\right\vert <\delta.
\]
Then%
\begin{align*}
\left\vert \varepsilon_{k}\right\vert  &  \leq\frac{1}{P^{\ast}\left(
s_{k}\right)  }\sum_{g\in E} f\left(  \xi_{go}\right)  e^{\left(
s_{k}-v\right)  \xi_{go}\left(  x\right)  }\left\vert \frac{h\left(
e^{d\left(  o,go\right)  }\right)  }{h\left(  e^{d\left(  x,go\right)
}\right)  }-1\right\vert e^{-s_{k}d\left(  x,go\right)  }h\left(  e^{d\left(
x,go\right)  }\right) \\
&  +\delta\frac{1}{P^{\ast}\left(  s_{k}\right)  }\sum_{g\in G\backslash E}
f\left(  \xi_{go}\right)  e^{\left(  s_{k}-v\right)  \xi_{go}\left(  x\right)
}e^{-s_{k}d\left(  x,go\right)  }h\left(  e^{d\left(  x,go\right)  }\right) \\
&  \leq\frac{1}{P^{\ast}\left(  s_{k}\right)  }\sum_{g\in E} f\left(  \xi
_{go}\right)  e^{\left(  s_{k}-v\right)  \xi_{go}\left(  x\right)  }\left\vert
\frac{h\left(  e^{d\left(  o,go\right)  }\right)  }{h\left(  e^{d\left(
x,go\right)  }\right)  }-1\right\vert e^{-s_{k}d\left(  x,go\right)  }h\left(
e^{d\left(  x,go\right)  }\right) \\
&  +\delta\int f\left(  \xi\right)  e^{\left(  s_{k}-v\right)  \xi\left(
x\right)  }d\nu_{x, s_{k}}\left(  \xi\right)  .
\end{align*}
Taking limit yields%
\[
\lim_{k\rightarrow\infty}\left\vert \varepsilon_{k}\right\vert \leq
\delta\|f\|_{\infty}e^{vd(o,x)} .
\]
Therefore $\lim\varepsilon_{k}=0$.
\end{proof}

We can integrate by parts along each $M_{0}\times\left\{  \xi\right\}  $, for
$\nu_{o}$-almost every $\xi$, and get for any function $f$ which is $C^{2}$
along the leaves of the lamination $\mathcal{W}$ and has a support contained
in $M_{0} \times\widehat{M}$:%
\begin{align*}
\int\Delta fd\nu &  =\int\left(  \int_{M_{0}}\Delta^{\mathcal{W}}%
fe^{-v\xi\left(  x\right)  }dx\right)  d\nu_{o}\left(  \xi\right) \\
&  =v\int\left(  \int_{M_{0}}\left\langle \nabla^{\mathcal{W}}f,\nabla
^{\mathcal{W}}\xi\right\rangle e^{-v\xi\left(  x\right)  }dx\right)  d\nu
_{o}\left(  \xi\right) \\
&  =v\int\left\langle \nabla^{\mathcal{W}}f,\nabla^{\mathcal{W}}%
\xi\right\rangle d\nu.
\end{align*}

The integral makes sense because $\nabla^{\mathcal{W}}\xi$ is defined Lebesgue
almost everywhere on the leaves and because, by (\ref{abs.cont}), the measure
$\nu$ has absolutely continuous conditional measures along the leaves
$\mathcal{W}$. By choosing the fundamental domain $M_{0}$, we get the same
formula for any function which is $C^{2}$ along the leaves of the lamination
$\mathcal{W}$ and has a small support. Using a partition of unity on $M$, we
see that for all functions on $X_{M}$ which are $C^{2}$ along the leaves of
the lamination $\mathcal{W}$, we have:%
\[
\int\Delta fd\nu=v\int\left\langle \nabla^{\mathcal{W}}f,\nabla^{\mathcal{W}%
}\xi\right\rangle d\nu.
\]

In the same way, one gets for all continuous functions $f_{1}, f_{2}$ which
are smooth along the leaves of the lamination $\mathcal{W}$:
\[
\int\mathrm{div} ^{\mathcal{W}}(f_{1} \nabla^{\mathcal{W}}f_{2}) d\nu= v \int
f_{1} \left\langle \nabla^{\mathcal{W}}f_{2},\nabla^{\mathcal{W}}%
\xi\right\rangle d\nu.
\]
By approximation, we have for all $\mathcal{W}$ vector field $Y$ which is
$C^{1}$ along the leaves and globally continuous,%

\begin{equation}
\int\mathrm{div}^{\mathcal{W}}Yd\nu=v\int\left\langle Y,\nabla^{\mathcal{W}%
}\xi\right\rangle d\nu. \label{cb}%
\end{equation}

Since the measure $\nu$ gives full measure to $\widetilde{M} \times
\partial\widehat{M} $, Theorem \ref{div} is proven.

\section{The rigidity theorem}

In this section we prove the rigidity theorem.

\begin{theorem}
\label{Rig2}Let $M^{n}$ be a compact Riemannian manifold with $Ric\geq-\left(
n-1\right)  $ and $\pi:\widetilde{M}\rightarrow M$ its universal covering.
Then the volume entropy of $\widetilde{M}$ satisfies $v\leq\left(  n-1\right)
$ and equality holds iff $M$ is hyperbolic.
\end{theorem}

Observe that this proves Theorem \ref{Rig}, since the volume entropy of the
universal covering is not smaller than the volume entropy of an intermediate
covering space. First we have

\begin{proposition}
\label{subh}For any $\xi\in\partial\widehat{M}$ we have $\Delta\left(
e^{-\left(  n-1\right)  \xi}\right)  \geq0$ in the sense of distribution.
\end{proposition}

\begin{proof}
It is well known that $\Delta\xi\leq n-1$ in the distribution sense for any
$\xi\in\partial\widehat{M}$. Indeed, suppose $\xi$ is given as in formula
(\ref{lim}). By the Laplacian comparison theorem
\[
\Delta f_{k}\left(  x\right)  \leq\left(  n-1\right)  \frac{\cosh\left(
d\left(  x,a_{k}\right)  \right)  }{\sinh\left(  d\left(  x,a_{k}\right)
\right)  }%
\]
in the distribution sense. Taking limit then yields $\Delta\xi\leq n-1$.
Therefore
\begin{align*}
\Delta\left(  e^{-\left(  n-1\right)  \xi}\right)   &  =-\left(  n-1\right)
e^{-\left(  n-1\right)  \xi}\left(  \Delta\xi-\left(  n-1\right)  \left\vert
\nabla\xi\right\vert ^{2}\right) \\
&  =-\left(  n-1\right)  e^{-\left(  n-1\right)  \xi}\left(  \Delta\xi-\left(
n-1\right)  \right) \\
&  \geq0,
\end{align*}
all understood in the sense of distribution.
\end{proof}

\bigskip

Let $p_{t}\left(  x,y\right)  $ be the heat kernel on $\widetilde{M}$. For any
function $f$ on $\widetilde{M}$ we define%
\[
P_{t}f\left(  x\right)  =\int_{\widetilde{M}}p_{t}\left(  x,y\right)  f\left(
y\right)  dy.
\]
We have $P_{t}\left(  g\cdot f\right)  =g\cdot P_{t}f$ for any $g\in G$. \ 

We now proceed to prove Theorem \ref{Rig2}. We consider the following vector
field on $\widetilde{M}\times\widehat{M}$
\[
Y_{t}\left(  x,\xi\right)  =\nabla\left(  P_{t}\xi\right)  \left(  x\right)
.
\]
It is easy to see that $Y_{t}$ descends to $X_{M}$, i.e. for any $g\in G$ we
have $Y_{t}\left(  gx,g\cdot\xi\right)  =g_{\ast}Y_{t}\left(  x,\xi\right)  $.
By Theorem \ref{div}%
\begin{align*}
v\int_{X_{M}}\left\langle \nabla^{w}\xi,Y_{t}\right\rangle d\nu &
=\int_{X_{M}}\mathrm{div}^{w}Y_{t}d\nu\\
&  =\int_{M}\left(  \int_{\partial\widehat{M}}\mathrm{div}^{w}Y_{t}%
e^{-v\xi\left(  x\right)  }d\nu_{o}\left(  \xi\right)  \right)  dx\\
&  =\int_{M}\left(  \int_{\partial\widehat{M}}\left(  \mathrm{div}^{w}\left(
Y_{t}e^{-v\xi\left(  x\right)  }\right)  +v\left\langle \nabla^{w}\xi
,Y_{t}\right\rangle e^{-v\xi\left(  x\right)  }\right)  d\nu_{o}\left(
\xi\right)  \right)  dx\\
&  =\int_{M}\left(  \int_{\partial\widehat{M}}\mathrm{div}^{w}\left(
Y_{t}e^{-v\xi\left(  x\right)  }\right)  d\nu_{o}\left(  \xi\right)  \right)
dx+v\int_{X_{M}}\left\langle \nabla^{w}\xi,Y_{t}\right\rangle d\nu,
\end{align*}
whence%
\[
\int_{M}\left(  \int_{\partial\widehat{M}}\mathrm{div}^{w}\left(
Y_{t}e^{-v\xi\left(  x\right)  }\right)  d\nu_{o}\left(  \xi\right)  \right)
dx=0.
\]
We now cover $M$ by finitely many open sets $\left\{  U_{i}:1\leq i\leq
k\right\}  $ s.t. each $U_{i}$ is so small that $\pi^{-1}\left(  U_{i}\right)
$ is the disjoint union of open sets each diffeomorphic to $U_{i}$ via $\pi$.
Let $\left\{  \chi_{i}\right\}  $ be a partition of unity subordinating to
$\left\{  U_{i}\right\}  $. For each $U_{i}$ let $\widetilde{U}_{i}$ be one of
the components of $\pi^{-1}\left(  U_{i}\right)  $ and let $\widetilde{\chi
}_{i}$ be the lifting of $\chi_{i}$ to $\widetilde{U}_{i}$. Then
\begin{align*}
0  &  =\int_{M}\left(  \int_{\partial\widehat{M}}\mathrm{div}^{w}\left(
Y_{t}e^{-v\xi\left(  x\right)  }\right)  d\nu_{o}\left(  \xi\right)  \right)
dx\\
&  =\sum_{i}\int_{M}\left(  \int_{\partial\widehat{M}}\mathrm{div}^{w}\left(
Y_{t}e^{-v\xi\left(  x\right)  }\right)  \chi_{i}\circ\pi\left(  x\right)
d\nu_{o}\left(  \xi\right)  \right)  dx\\
&  =\sum_{i}\int_{U_{i}}\left(  \int_{\partial\widehat{M}}\mathrm{div}%
^{w}\left(  Y_{t}e^{-v\xi\left(  x\right)  }\right)  \chi_{i}\circ\pi\left(
x\right)  d\nu_{o}\left(  \xi\right)  \right)  dx\\
&  =\sum_{i}\int_{\widetilde{U}_{i}}\left(  \int_{\partial\widehat{M}%
}\mathrm{div}^{w}\left(  Y_{t}e^{-v\xi\left(  x\right)  }\right)  \chi
_{i}\circ\pi\left(  x\right)  d\nu_{o}\left(  \xi\right)  \right)  dx\\
&  =\sum_{i}\int_{\partial\widehat{M}}\left(  \int_{\widetilde{U}_{i}%
}\mathrm{div}^{w}\left(  Y_{t}e^{-v\xi\left(  x\right)  }\right)
\widetilde{\chi}_{i}dx\right)  d\nu_{o}\left(  \xi\right) \\
&  =-\sum_{i}\int_{\partial\widehat{M}}\left(  \int_{\widetilde{U}_{i}%
}\left\langle Y_{t},\nabla\widetilde{\chi}_{i}\right\rangle e^{-v\xi\left(
x\right)  }dx\right)  d\nu_{o}\left(  \xi\right)  .
\end{align*}
Letting $t\rightarrow0$ yields%
\[
\sum_{i}\int_{\partial\widehat{M}}\left(  \int_{\widetilde{U}_{i}}\left\langle
\nabla\xi,\nabla\widetilde{\chi}_{i}\right\rangle e^{-v\xi\left(  x\right)
}dx\right)  d\nu_{o}\left(  \xi\right)  =0.
\]
Integrating by parts again, we obtain%
\begin{align*}
&  \sum_{i}\int_{\partial\widehat{M}}\left(  \int_{\widetilde{U}_{i}}%
e^{-v\xi\left(  x\right)  }\Delta\widetilde{\chi}_{i}dx\right)  d\nu
_{o}\left(  \xi\right) \\
&  =-\sum_{i}\int_{\partial\widehat{M}}\left(  \int_{\widetilde{U}_{i}%
}\left\langle \nabla\left(  e^{-v\xi\left(  x\right)  }\right)  ,\nabla
\widetilde{\chi}_{i}\right\rangle dx\right)  d\nu_{o}\left(  \xi\right) \\
&  =v\sum_{i}\int_{\partial\widehat{M}}\left(  \int_{\widetilde{U}_{i}%
}\left\langle \nabla\xi,\nabla\widetilde{\chi}_{i}\right\rangle e^{-v\xi
\left(  x\right)  }dx\right)  d\nu_{o}\left(  \xi\right)  .
\end{align*}
Therefore
\[
\int_{\partial\widehat{M}}\sum_{i}\left(  \int_{\widetilde{U}_{i}}%
e^{-v\xi\left(  x\right)  }\Delta\widetilde{\chi}_{i}dx\right)  d\nu
_{o}\left(  \xi\right)  =0.
\]
We now assume $v=n-1$. By Proposition \ref{subh} $\Delta e^{-v\xi\left(
x\right)  }\geq0$ in the sense of distribution for all $\xi\in\partial
\widehat{M}$ and hence $\int_{\widetilde{U}_{i}}e^{-v\xi\left(  x\right)
}\Delta\widetilde{\chi}_{i}dx\geq0$ for all $i$. Therefore we conclude for
$\nu_{o}$-a.e. $\xi\in\partial\widehat{M}$%
\[
\int_{\widetilde{U}_{i}}e^{-v\xi\left(  x\right)  }\Delta\widetilde{\chi}%
_{i}dx=0\text{ }%
\]
for all $i$. In this discussion we can replace $\widetilde{U}_{i}$ by
$g\widetilde{U}_{i}$ \ and $\widetilde{\chi}_{i}$ by $g\cdot\widetilde{\chi
}_{i}$ for any $g\in G$. Since $G$ is countable we conclude for $\nu_{o}$-a.e.
$\xi\in\widehat{M}$%
\[
\int_{g\widetilde{U}_{i}}e^{-v\xi\left(  x\right)  }\Delta\left(
g\cdot\widetilde{\chi}_{i}\right)  dx=0
\]
for all $i$ and $g\in G$.

We claim that $\Delta e^{-v\xi\left(  x\right)  }=0$ in the sense of
distribution. Indeed, denote the distribution $\Delta e^{-v\xi\left(
x\right)  }$ simply by $T$, i.e. $T\left(  f\right)  =\int_{\widetilde{M}%
}e^{-v\xi\left(  x\right)  }\Delta f\left(  x\right)  dx$ for $f\in
C_{c}^{\infty}\left(  \widetilde{M}\right)  $. We know that $T\left(
f\right)  \geq0$ if $f\geq0$. We observe that $\left\{  g\cdot\widetilde{\chi
}_{i}:1\leq i\leq k,g\in G\right\}  $ is a partition of unity on
$\widetilde{M} $ subordinating to the open cover$\left\{  g\widetilde{U}%
_{i}:1\leq i\leq k,g\in G\right\}  $. Hence for any $f\in C_{c}^{\infty
}\left(  \widetilde{M}\right)  $ with $f\geq0$ we have%
\[
0\leq f\leq C\sum_{g\widetilde{U}_{i}\cap\mathrm{spt}f\neq\emptyset}%
g\cdot\widetilde{\chi}_{i},
\]
with $C=\sup f$. Notice that the right hand side is a finite sum as the
support $\mathrm{spt}f$ is compact. Then
\begin{align*}
0  &  \leq T\left(  f\right)  \leq T\left(  C\sum_{g\widetilde{U}_{i}%
\cap\mathrm{spt}f}g\cdot\widetilde{\chi}_{i}\right) \\
&  =C\sum_{g\widetilde{U}_{i}\cap\mathrm{spt}f\neq\emptyset}T\left(
g\cdot\widetilde{\chi}_{i}\right) \\
&  =0.
\end{align*}
Hence $T\left(  f\right)  =0$, i.e. $\Delta e^{-v\xi\left(  x\right)  }=0$ in
the sense of distribution. By elliptic regularity $\phi=e^{-v\xi\left(
x\right)  }$ is then a smooth harmonic function and obviously $\left\vert
\nabla\log\phi\right\vert =n-1$. The rigidity now follows from the following result.

\begin{theorem}
Let $N^{n}$ be a complete, simply connected Riemannian manifold s.t.

\begin{enumerate}
\item $Ric\geq-\left(  n-1\right)  $;

\item the sectional curvature is bounded.
\end{enumerate}

Suppose that there is a positive harmonic function $\phi$ on $N$ s.t.
$\left\vert \nabla\log\phi\right\vert =n-1$. Then $N$ is isometric to the
hyperbolic space $\mathbb{H}^{n}$.
\end{theorem}

\begin{remark}
Without assuming bounded sectional curvature, the second author \cite{W}
proved that $N$ is isometric to the hyperbolic space $\mathbb{H}^{n}$ provided
that there are two such special harmonic functions. We thank Ovidiu Munteanu
for pointing out that one such special harmonic function is enough if the
sectional curvature is bounded.
\end{remark}

\begin{proof}
The first step is to show that $\phi$ satisfies an over-determined system
which then leads to the splitting of $N$ as a warped product. This is standard
and we outline the argument. Let $f=\log\phi$. We have $\Delta f=-\left\vert
\nabla f\right\vert ^{2}=-\left(  n-1\right)  ^{2}$. Since
\[
D^{2}f\left(  \nabla f,\nabla f\right)  =\frac{1}{2}\left\langle \nabla
f,\nabla\left\vert \nabla f\right\vert ^{2}\right\rangle =0,
\]
we have by Cauchy-Schwarz
\begin{equation}
\left\vert D^{2}f\right\vert ^{2}\geq\frac{\left(  \Delta f\right)  ^{2}}%
{n-1}=\left(  n-1\right)  ^{3} \label{kato}%
\end{equation}
with equality iff
\[
D^{2}f=-\left(  n-1\right)  \left[  g-\frac{1}{\left(  n-1\right)  ^{2}%
}df\otimes df\right]
\]

On the other hand, by the Bochner formula, we have%
\begin{align*}
0  &  =\frac{1}{2}\Delta\left\vert \nabla f\right\vert ^{2}\\
&  =\left\vert D^{2}f\right\vert ^{2}+\left\langle \nabla f,\nabla\Delta
f\right\rangle +Ric\left(  \nabla f,\nabla f\right) \\
&  \geq\left\vert D^{2}f\right\vert ^{2}-\left(  n-1\right)  \left\vert \nabla
f\right\vert ^{2}\\
&  =\left\vert D^{2}f\right\vert ^{2}-\left(  n-1\right)  ^{3},
\end{align*}
i.e. $\left\vert D^{2}f\right\vert ^{2}\leq\left(  n-1\right)  ^{3}$. This
show that (\ref{kato}) is in fact an equality. Therefore we have
$D^{2}f=-\left(  n-1\right)  \left[  g-\frac{1}{\left(  n-1\right)  ^{2}%
}df\otimes df\right]  $. \ From this one can show that $N=\mathbb{R\times
}\Sigma^{n-1}$ with the metric $g=dt^{2}+e^{2t}h$, where $h$ is a Riemannian
metric on $\Sigma$. For more detail, cf. \cite{LW3}.

For any $p\in\Sigma$ let $\left\{  e_{i}\right\}  $ be an orthogonal basis on
$\left(  T_{p}\Sigma,h\right)  $. By a simple calculation using the Gauss
equation the curvature of $N$ is given by%
\[
R\left(  e^{-t}e_{i},e^{-t}e_{j},e^{-t}e_{i},e^{-t}e_{j}\right)  =e^{-4t}%
R^{h}\left(  e_{i},e_{j},e_{i},e_{j}\right)  -\left(  \delta_{ii}\delta
_{jj}-\delta_{ij}^{2}\right)  ,
\]
where $R^{h}$ is the curvature tensor of $\left(  \Sigma,h\right)  $. Since
$N$ has bounded sectional curvature, the left hand side is bounded in $t$.
Therefore $R^{h}=0$, i.e. $\Sigma$ is flat. Since $\Sigma$ is also simply
connected as $N$ is simply connected, it is isometric to $\mathbb{R}^{n-1}$.
It follows that $N$ is the hyperbolic space.
\end{proof}

\section{\bigskip The K\"{a}hler and quaternionic K\"{a}hler cases}

In this Section, we first discuss the K\"{a}hler case.

\begin{theorem}
\label{Kan}Let $M$ be a compact K\"{a}hler manifold with $\dim_{\mathbb{C}%
}M=m$ and $\pi:\widetilde{M}\rightarrow M$ the universal covering. If the
bisectional curvature $K_{\mathbb{C}}\geq-2$, then the volume entropy $v$
satisfies $v\leq2m$. Moreover equality holds iff $M$ is complex hyperbolic
(normalized to have constant holomorphic sectional curvature $-4$).
\end{theorem}

The inequality follows from the comparison theorem in \cite{LW2}. Indeed,
under the curvature assumption $K_{\mathbb{C}}\geq-2$, Li and J. Wang
\cite{LW2} proved
\[
\mathrm{vol}B_{\widetilde{M}}\left(  x,r\right)  \leq V_{\mathbb{CH}^{m}%
}\left(  r\right)  =\tau_{2m-1}\int_{0}^{r}\sinh\left(  2t\right)
\sinh^{2\left(  m-1\right)  }\left(  t\right)  dt,
\]
where $\tau_{2m-1}$ is the volume of the unit sphere in $\mathbb{R}^{2m}$. It
follows that $v\leq2m$. Another consequence of the comparison theorem is that
for $\xi\in\partial\widehat{M}$%
\[
\Delta\xi\leq2m
\]
in the distribution sense. It follows as in the Riemannian case that $\Delta
e^{-2m\xi}\geq0$ in the distribution sense.

We now assume that $v=2m$. By the argument in Section 3 we conclude that $\xi$
is smooth and%
\[
\Delta\xi=2m,\left\vert \nabla\xi\right\vert =1
\]
for $\nu_{o}$-a.e. $\xi\in\partial\widehat{M}$. Take such a function $\xi$. We
choose a local unitary frame $\left\{  X_{i},\overline{X}_{i}\right\}  $.

\begin{lemma}
We have%
\[
\xi_{i\overline{j}}=\delta_{ij},\xi_{ij}=-2\xi_{i}\xi_{j}.
\]

\end{lemma}

\begin{proof}
Without loss of generality we can assume that $X_{1}=\left(  \nabla\xi
-\sqrt{-1}J\nabla\xi\right)  /\sqrt{2}$. Therefore%
\[
\xi_{1}=\frac{1}{\sqrt{2}},\xi_{i}=0\text{ for }i\geq2.
\]
Suppose $\xi$ is given as in (\ref{lim}). Let $p\in\widetilde{M}$ and we use
the construction preceding Lemma \ref{BF}. By the second part of that Lemma we
see that for any $s>0$ the function $u_{s}\left(  x\right)  =\xi\left(
p\right)  +d\left(  x,\gamma\left(  s\right)  \right)  -s$ is a support
function for $\xi$ from above at $p$. Moreover $u_{s}$ is clearly smooth at
$p$. Therefore we have at $p$%
\[
D^{2}\xi\leq D^{2}u_{s}.
\]
By the comparison theorem in \cite{LW2} we have%
\begin{align*}
\xi_{1\overline{1}}  &  \leq\left(  u_{s}\right)  _{1\overline{1}}\leq
\frac{\cosh2s}{\sinh2s},\\
\xi_{i\overline{i}}  &  \leq\left(  u_{s}\right)  _{i\overline{i}}\leq
\frac{\cosh s}{\sinh s}\text{ for }i\geq2.
\end{align*}
Taking limit as $s\rightarrow\infty$ yields $\xi_{i\overline{i}}\leq1.$On the
other hand we have $\sum_{i=1}^{m}\xi_{i\overline{i}}=\frac{1}{2}\Delta\xi=m$.
Therefore we must have
\begin{equation}
\xi_{i\overline{i}}=1. \label{hei}%
\end{equation}
By the Bochner formula we have%
\begin{align*}
0  &  =\frac{1}{2}\Delta\left\vert \nabla\xi\right\vert ^{2}=\left\vert
D^{2}\xi\right\vert ^{2}+\left\langle \nabla\xi,\nabla\Delta\xi\right\rangle
+Ric\left(  \nabla\xi,\nabla\xi\right) \\
&  \geq\left\vert D^{2}\xi\right\vert ^{2}-2\left(  m+1\right)  .
\end{align*}
Therefore%
\begin{equation}
\left\vert \xi_{i\overline{j}}\right\vert ^{2}+\left\vert \xi_{ij}\right\vert
^{2}\leq m+1. \label{fhe}%
\end{equation}
We have
\[
0\leq\left\vert \xi_{ij}+2\xi_{i}\xi_{j}\right\vert ^{2}=\left\vert \xi
_{ij}\right\vert ^{2}+2\xi_{ij}\xi_{\overline{i}}\xi_{\overline{j}}%
+2\xi_{\overline{i}\overline{j}}\xi_{i}\xi_{j}+1.
\]
By differentiating $\xi_{j}\xi_{\overline{j}}=\frac{1}{2}$ we obtain $\xi
_{ij}\xi_{\overline{j}}+\xi_{j}\xi_{i\overline{j}}=0,\xi_{\overline{i}j}%
\xi_{\overline{j}}+\xi_{j}\xi_{\overline{i}\overline{j}}=0$. Hence from the
previous inequality we obtain%
\begin{equation}
\left\vert \xi_{ij}\right\vert ^{2}\geq4\xi_{i\overline{j}}\xi_{\overline{i}%
}\xi_{j}-1=2\xi_{1\overline{1}}-1. \label{ho}%
\end{equation}
On the other hand%
\[
\left\vert \xi_{i\overline{j}}\right\vert ^{2}\geq\frac{1}{m}\left(  \frac
{1}{2}\Delta\xi\right)  ^{2}=m
\]
Combining this inequality with (\ref{fhe}) and (\ref{ho}) yields
\[
\xi_{1\overline{1}}\leq1.
\]
However we already proved that equality holds (\ref{hei}). By inspecting the
argument we conclude that $\xi$ satisfies the following over-determined system%
\[
\xi_{i\overline{j}}=\delta_{ij},\xi_{ij}=-2\xi_{i}\xi_{j}.
\]

\end{proof}

With such a function, it is proved by Li and J. Wang \cite{LW1} that
$\widetilde{M}$ is isometric to $\mathbb{R\times}N^{2m-1}$ with the metric%
\[
g=dt^{2}+e^{-4t}\theta_{0}^{2}+e^{-2t}\sum_{i=1}^{2\left(  m-1\right)  }%
\theta_{i}^{2},
\]
where $\left\{  \theta_{0},\theta_{1},\cdots,\theta_{2\left(  m-1\right)
}\right\}  $ is an orthonormal frame for $T^{\ast}N$. Moreover, since our
$\widetilde{M}$ is simply connected and has bounded curvature, $N$ is
isometric to the Heisenberg group by their theorem. Therefore $\widetilde{M}$
is isometric to the complex hyperbolic space $\mathbb{CH}^{m}$.

Theorem \ref{QK} for quaternionic K\"{a}hler manifolds is proved in the same
way, using the work of Kong, Li and Zhou \cite{KLZ} in which they proved a
Laplacian comparison theorem for quaternionic K\"{a}hler manifolds.

We close with some remarks. An obvious question is whether this theorem
remains true if the curvature condition is relaxed to $Ric\geq-2\left(
m+1\right)  $. This seems a very subtle question. It is quite unlikely that
the comparison theorem for K\"{a}hler manifolds could still hold in this case.
On the other hand it is conceivable Theorem \ref{Kan} will remain valid due to
some global reason. This hope is partly based on the recent work of Munteanu
\cite{Mu} in which a sharp estimate for the Kaimanovich entropy is derived
under the condition $Ric\geq-2\left(  m+1\right)  $.

\end{document}